\documentclass{amsart}
\usepackage{amsfonts}
\usepackage{amssymb}
\usepackage{amsmath,amsxtra,amssymb,mathrsfs,amscd}
\numberwithin{equation}{section}
\usepackage[all]{xy}
\usepackage{bbm} 

\newtheorem{thm}{\textbf{Theorem}}[section]
\newtheorem{prop}[thm]{\textbf{Proposition}}
\newtheorem{coro}[thm]{\textbf{Corollary}}
\newtheorem{lemma}[thm]{\textbf{Lemma}}

\theoremstyle{definition}

\theoremstyle{definition}
\newtheorem{defi}[thm]{\textbf{Definition}}

\theoremstyle{definition}
\newtheorem{remark}[thm]{\textbf{Remark}}

\theoremstyle{definition}

\newcommand{\injtp}{\overset{~\vee_p}{\otimes}}
\newcommand{\projtp}{\overset{~\wedge_p}{\otimes}}

\newcommand{\cbp}{\mathcal{CB}_p}
\newcommand{\ptp}{\overset{\wedge_p}{\otimes}}
\newcommand{\itp}{\overset{\vee_p}{\otimes}}

\newcommand{\llangle}{\langle\!\langle}
\newcommand{\rrangle}{\rangle\!\rangle}

\def\proclaim #1. #2\par{\medbreak
\noindent{\bf#1.\enspace}{\sl#2}\par\medbreak}
\allowdisplaybreaks

\begin{document}

\title[Conditions $C_p$, $C'_p$, and $C''_p$ for $p$-operator spaces]
{Conditions $C_p$, $C'_p$, and $C''_p$ for $p$-operator spaces}
\author{Jung-Jin Lee}
\address{Department of Mathematics and Statistics\\
Mount Holyoke College, South Hadley, MA 01075, USA} \email[Jung-Jin
Lee]{jjlee@mtholyoke.edu}
\thanks{The author was supported by Hutchcroft Fund, Department of Mathematics and Statistics, Mount Holyoke College}


\date{\today}


\begin{abstract}
Conditions $C$, $C'$, and $C''$ were introduced for operator spaces in an attempt to study local reflexivity and exactness of operator spaces \cite[Chapter 14]{EffrosRuan}. For example, it is known that an operator space $W$ is locally reflexive if and only if $W$ satisfies condition $C''$ \cite[Theorem 14.3.1]{EffrosRuan} and  an operator space $V$ is exact if and only if $V$ satisfies condition $C'$ \cite[Theorem 14.4.1]{EffrosRuan}. It is also known that an operator space $V$ satisfies condition $C$ if and only if it satisfies conditions $C'$ and $C''$ \cite[Lemma 14.2.1]{EffrosRuan}, \cite[Theorem 5]{Han}. In this paper, we define $p$-operator space analogues of these definitions, which will be called conditions $C_p$, $C'_p$, and $C''_p$, and show that a $p$-operator space on $L_p$ space satisfies condition $C_p$ if and only if it satisfies both conditions $C'_p$ and $C''_p$. The $p$-operator space injective tensor product of $p$-operator spaces will play a key role.
\end{abstract}

\maketitle

\section{Introduction to $p$-Operator Spaces}
A \textit{concrete operator space} $V$ is defined to be a closed subspace of $\mathcal{B}(H)$, where $\mathcal{B}(H)$ denotes the space of all bounded linear operators on a Hilbert space $H$. For each $n\in \mathbb{N}$, the matrix algebra $\mathbb{M}_n(\mathcal{B}(H))$ with entries in $\mathcal{B}(H)$ can be identified with $\mathcal{B}(\underbrace{H\oplus\cdots\oplus H}_{n})$ via \textit{matrix multiplication} 
$$\left[\begin{array}{ccc} & & \\ & T_{ij} & \\ & & \end{array}\right]\left[\begin{array}{c} h_1 \\ \vdots \\ h_n \end{array}\right]=\left[\begin{array}{c} \sum_{j=1}^n T_{1j}h_j \\ \vdots \\ \sum_{j=1}^n T_{nj}h_j \end{array}\right],\quad [T_{ij}]\in \mathbb{M}_n(\mathcal{B}(H)),\quad h_j\in H,$$
and this gives rise to a norm $\|\cdot\|_n$ on $\mathbb{M}_n(V)$, which we denote by $M_n(V)$. It is then easy to verify that the following two properties (called Ruan's axioms) hold:
\begin{description}
\item[$\mathcal{D}_\infty$] for $u \in M_n(V)$ and $v \in M_m(V)$, we have $\left\|\left[\begin{array}{cc} u & 0 \\0 & v \end{array}\right]\right\|_{n+m}=\max\{\|u\|_n,\|v\|_m\}$.
\item[$\mathcal{M}$] for $u \in M_m(V)$, $\alpha \in \mathbb{M}_{n,m}(\mathbb{C})$, and $\beta \in \mathbb{M}_{m,n}(\mathbb{C})$, we have
$\|\alpha u \beta\|_n\leq \|\alpha\|\|u\|_m\|\beta\|$, where $\|\alpha\|$ is the norm of $\alpha$ as a member of $\mathcal{B}(\ell_2^m,\ell_2^n)$, and similarly for $\beta$.
\end{description}

An \textit{abstract operator space} is a Banach space $X$ together with a family of norms $\|\cdot\|_n$ defined on $\mathbb{M}_n(X)$ satisfying the conditions $\mathcal{D}_\infty$ and $\mathcal{M}$ above. In \cite{Ruan}, Ruan showed that these two concepts coincide and after Ruan's characterization, operator space theory has really been taken off and quickly developed into an active research area in modern analysis. Many important applications have been found in some related areas. For example, let $G$ be a locally compact group. It is well known that $G$ is amenable if and only if the convolution algebra $L_1(G)$ is amenable as a Banach algebra \cite{Johnson}. We consider another Banach algebra called the \textit{Fourier algebra} $A(G)$ which consists of all coefficient functions of the left regular representation $\lambda$ of $G$, i.e.,
$$A(G)=\{\omega(\cdot)=\langle \lambda(\cdot)\xi, \eta\rangle: \xi,\eta\in L_2(G)\}.$$

By \cite{Eymard}, $A(G)$ is a commutative Banach algebra with respect to pointwise multiplication and can be regarded as the predual of $VN(G)$, the group von Neumann algebra of $G$. If $G$ is abelian, then its dual group $\hat{G}$ is also abelian and we have the isometric isomorphism $A(G)\cong L_1(\hat{G})$, and this suggests a relationship between the amenability of $G$ and the amenability (as a Banach algebra) of $A(G)$. Indeed, if $A(G)$ is amenable, then $G$ is amenable. In the opposite direction, Johnson showed that the Banach algebra $A(G)$ fails to be amenable even in the case of very simple compact groups, such as $SU(2,\mathbb{C})$ \cite{Johnson_Fourier_algebra}.

In \cite{Ruan_operator_amenability}, Ruan studied the \textit{operator amenability} of $A(G)$ which can be regarded as the amenability of $A(G)$ in the category of operator spaces, and proved that a locally compact group $G$ is amenable if and only if $A(G)$ is operator amenable. This suggests that $A(G)$ is better viewed as an operator space, and motivated by this observation, there has been some research \cite{Daws, ALR} to study \textit{Fig\`{a}-Talamanca-Herz Algebra} $A_p(G)$, which can be regarded as an $L_p$ space generalization of the Fourier algebra $A(G)$ (The reader is referred to \cite{Figa_Talamanca, Herz} for more details on $A_p(G)$), in the framework of $L_p$ space generalization of operator spaces. This leads to the definition of $p$-operator spaces we will give below. Throughout this paper, we let $1<p<\infty$.

\begin{defi} Let $SQ_p$ denote the collection of subspaces of quotients of $L_p$ spaces. A Banach space $X$ is called a \textit{concrete
$p$-operator space} if $X$ is a closed subspace of $\mathcal{B}(E)$ for some $E\in SQ_p$, where $\mathcal{B}(E)$ denotes the space of all bounded linear operators on $E$.
\end{defi}

Let $\mathbb{M}_n(X)$ denote the linear space of all $n\times n$ matrices with entries in $X$. For a concrete $p$-operator space $X\subseteq\mathcal{B}(E)$ and for each $n\in \mathbb{N}$, define a norm $\|\cdot\|_n$ on
$\mathbb{M}_n(X)$ by identifying $\mathbb{M}_n(X)$ as a subspace of $\mathcal{B}(\ell_p^n(E))$, and let $M_n(X)$ denote the corresponding normed space. The norms $\|\cdot\|_n$ then satisfy

\begin{description}
\item[$\mathcal{D}_\infty$] for $u \in M_n(X)$ and $v \in M_m(X)$, we have $\|u\oplus v\|_{n+m}=\max\{\|u\|_n,\|v\|_m\}$.
\item[$\mathcal{M}_p$] for $u \in M_m(X)$, $\alpha \in \mathbb{M}_{n,m}(\mathbb{C})$, and $\beta \in \mathbb{M}_{m,n}(\mathbb{C})$, we have
$\|\alpha u \beta\|_n\leq \|\alpha\|\|u\|_m\|\beta\|$, where $\|\alpha\|$ is the norm of $\alpha$ as a member of $\mathcal{B}(\ell_p^m,\ell_p^n)$, and similarly for $\beta$.
\end{description}

\begin{remark}
When $p=2$, these are Ruan's axioms and $2$-operator spaces are simply operator spaces because the $SQ_2$ spaces are exactly Hilbert spaces.
\end{remark}

As in operator spaces, we can also define abstract $p$-operator spaces.

\begin{defi}
An \textit{abstract $p$-operator space} is a Banach space $X$ together with a sequence of norms $\|\cdot\|_n$ defined on $\mathbb{M}_n(X)$ satisfying the conditions $\mathcal{D}_\infty$ and $\mathcal{M}_p$ above.
\end{defi}

Thanks to the following theorem by Le Merdy, we do not distinguish between concrete $p$-operator spaces and abstract $p$-operator spaces, so from now on we will merely speak of $p$-operator spaces.

\begin{thm}\cite[Theorem 4.1]{LeMerdy} \label{LeMerdy characterization} An abstract $p$-operator space $X$ can be isometrically embedded in $\mathcal{B}(E)$ for some $E\in SQ_p$ in such a way that the canonical norms on $\mathbb{M}_n(X)$ arising from this embedding agree with the given norms.
\end{thm}

Note that a linear map $u:X\to Y$ between $p$-operator spaces $X$ and $Y$ induces a map $u_n:M_n(X)\to M_n(Y)$ by applying $u$ entrywise. We say that $u$ is $p$-\textit{completely bounded} if $\|u\|_{pcb}:=\sup_n\|u_n\|<\infty$. Similarly, we define $p$-\textit{completely contractive}, $p$-\textit{completely isometric}, and $p$-\textit{completely quotient} maps. We write $\mathcal{CB}_p(X,Y)$ for the space of all $p$-completely bounded maps from $X$ into $Y$, and to turn the mapping space $\mathcal{CB}_p(X,Y)$ into a $p$-operator space, we define a norm on $\mathbb{M}_n(\mathcal{CB}_p(X,Y))$ by identifying this space with $\mathcal{CB}_p(X,M_n(Y))$. Using Le Merdy's theorem, one can show that $\mathcal{CB}_p(X,Y)$ itself is a $p$-operator space. In particular, the $p$-\textit{operator dual space} of $X$ is defined to be $\mathcal{CB}_p(X,\mathbb{C})$. The next lemma by Daws shows that we may identify the Banach dual space $X'$ of $X$ with the $p$-operator dual space $\mathcal{CB}_p(X,\mathbb{C})$ of $X$.

\begin{lemma}\cite[Lemma 4.2]{Daws} \label{banach dual = operator dual}
Let $X$ be a $p$-operator space, and let $\varphi\in X'$, the Banach dual of $X$. Then $\varphi$ is $p$-completely bounded as a map to $\mathbb{C}$. Moreover, $\|\varphi\|_{pcb}=\|\varphi\|$.
\end{lemma}

If $E=L_p(\mu)$ for some measure $\mu$ and $X\subseteq \mathcal{B}(E)=\mathcal{B}(L_p(\mu))$, then we say that $X$ is a $p$-\textit{operator space on} $L_p$ \textit{space}. These $p$-operator spaces are often easier to work with. For example, let $\kappa_X:X\to X''$ denote the canonical inclusion from a $p$-operator space $X$ into its second dual. Contrary to operator spaces, $\kappa_X$ is \textit{not} always $p$-completely isometric. Thanks to the following theorem by Daws, however, we can easily characterize those $p$-operator spaces with the property that the canonical inclusion is $p$-completely isometric.

\begin{prop}\cite[Proposition 4.4]{Daws} \label{Daws characterization}
Let $X$ be a $p$-operator space. Then $\kappa_X$ is a $p$-complete contraction. Moreover, $\kappa_X$ is a $p$-complete isometry if and only if $X\subseteq \mathcal{B}(L_p(\mu))$ $p$-completely isometrically for some measure $\mu$.
\end{prop}

Conditions $C$, $C'$, and $C''$ for operator spaces were introduced and studied in \cite[Chapter 14]{EffrosRuan} and \cite{Han} and they play an important role in understanding local reflexivity and exactness of operator spaces. For example, it is known that an operator space is locally reflexive if and only if it satisfies condition $C''$ \cite[Theorem 14.3.1]{EffrosRuan}. It is also known that an operator space is exact if and only if it satisfies condition $C'$ \cite[Theorem 14.4.1]{EffrosRuan}. In this paper, we define $p$-operator space analogues of these conditions, which will be called conditions $C_p$, $C'_p$, and $C''_p$, and show that a $p$-operator space on $L_p$ space satisfies condition $C_p$ if and only if it satisfies both conditions $C'_p$ and $C''_p$. 

\section{Tensor Product of $p$-Operator Spaces}

In this section, we recall basic properties of tensor products on $p$-operator spaces studied in \cite{Daws, ALR}. We mainly focus on $p$-projective tensor product and $p$-injective tensor product.
\begin{defi}
Let $X,Y$ be $p$-operator spaces. Let $X\otimes Y$ denote the algebraic tensor product of $X$ and $Y$. For $u\in \mathbb{M}_n (X\otimes Y)$, let
$$\|u\|_{\wedge_p}=\inf\{\|\alpha\|\|v\|\|w\|\|\beta\|:u=\alpha(v\otimes w)\beta\},$$
where the infimum is taken over $r,s\in \mathbb{N}$, $\alpha\in M_{n,r\times s}$, $v\in M_{r}(X)$, $w\in M_{s}(Y)$, and $\beta \in M_{r\times s,n}$.
\end{defi}

Daws defined and studied the $p$-projective tensor product \cite{Daws}. Note that $\|\cdot\|_{\wedge_p}$ gives the algebraic tensor product $X\otimes Y$ a $p$-operator space structure \cite[Proposition 4.8]{Daws}. Furthermore, $\|\cdot\|_{\wedge_p}$ is the largest subcross $p$-operator space norm on $X\otimes Y$ in the sense that $\|x\otimes y\|\leq \|x\|_r\|y\|_s$ for all $x\in M_r(X)$ and all $y\in M_s(Y)$ \cite[Proposition 4.8]{Daws}. The $p$-\textit{operator space projective tensor product} is defined to be the completion of $X\otimes Y$ with respect to this norm and is denoted by $X\projtp Y$.

\begin{remark} \label{banach proj p op proj}\quad
\begin{enumerate}
\item One can show that $p$-operator space projective tensor product is commutative, i.e., $X\projtp Y=Y\projtp X$ $p$-completely isometrically.
\item By universality of the Banach space projective tensor product $\overset{\pi}{\otimes}$ \cite[A.3.3]{BlecherLeMerdy}, we have
$$\|u\|_{\wedge_p}\leq \|u\|_{\pi}$$
for all $u\in X\otimes Y$.
\end{enumerate}
\end{remark}

Let $V,W$, and $Z$ be $p$-operator spaces, and let $\psi:V\times W\to Z$ be a bilinear
map. Define bilinear maps $\psi_{r,s;t,u}$ by
$$\psi_{r,s;t,u}:M_{r,s}(V)\times M_{t,u}(W)\to M_{r\times t,s\times u}(Z),\qquad (v,w)\mapsto (\psi(v_{i,j},w_{k,l})),$$
and let $\psi_{r;s}=\psi_{r,r;s,s}$. Finally define
$$\|\psi\|_{jpcb}=\sup\{\|\psi_{r;s}\|:r,s\in \mathbb{N}\}.$$

We say that $\psi$ is \textit{jointly $p$-completely bounded} (respectively,
\textit{jointly $p$-completely contractive}) if
$\|\psi\|_{jpcb}<\infty$ (respectively, $\|\psi\|_{jpcb}\leq 1$). The space
of all jointly $p$-completely bounded maps from $V\times W$ to $Z$
will be denoted by $\cbp(V\times W,Z)$ and this space can be turned
into a $p$-operator space in the same way as for $\cbp(V,W)$. Here we collect some results on the $p$-projective tensor product for convenience.

\begin{prop}\cite[Proposition 4.9]{Daws} \label{projtp and dual}
Let $X,Y$, and $Z$ be $p$-operator spaces. Then we have natural $p$-completely isometric identifications $$\mathcal{CB}_p(X\projtp Y, Z)=\cbp(X\times Y,Z)=\mathcal{CB}_p(X,\mathcal{CB}_p(Y,Z)).$$ In particular,
$$(X\projtp Y)'=\mathcal{CB}_p(X,Y').$$
\end{prop}

As in operator spaces, the $p$-operator space projective tensor product is projective in the following sense:

\begin{prop}\cite[Proposition 4.10]{Daws} Let $X, X_1, Y$, and $Y_1$ be $p$-operator spaces. If $u:X\to X_1$ and $v:Y\to Y_1$ are $p$-complete quotient maps, then $u\otimes v$ extends to a $p$-complete quotient map $u\otimes v:X\projtp Y \to X_1\projtp Y_1$.
\end{prop}

We now briefly introduce the $p$-operator space injective tensor product.
\begin{defi}\label{definition of itp} Let $X,Y$ be $p$-operator spaces. Regarding the algebraic tensor product $X\otimes Y$ as a subspace of $\cbp(X',Y)$, we define the $p$-\textit{operator space injective tensor product} $X\injtp Y$ to be the completion of $X\otimes Y$ in $\cbp(X',Y)$.
\end{defi}

To be precise, for $u=[u_{ij}]\in \mathbb{M}_n(X\otimes Y)$ with $u_{ij}=\sum_{k=1}^{N_{ij}}x_k^{ij}\otimes y_k^{ij}$, the $p$-operator space injective tensor product norm $\|u\|_{\vee_p}$ is defined by 
\begin{align}
\label{defi of injtp}  \begin{split}
    \|u\|_{\vee_p}=&\|u\|_{M_n(\cbp(X',Y))}=\|u\|_{\cbp(X',M_n(Y))}\\
     =&\sup\left\{\left\|\left[ \sum_{k=1}^{N_{ij}} f_{st}(x_k^{ij})y_k^{ij}\right]\right\|_{M_{mn}(Y)}:m\in \mathbb{N}, f=[f_{st}]\in M_m(X')_1\right\},
  \end{split}
\end{align}
where $M_m(X')_1$ denotes the closed unit ball of $M_m(X')=\cbp(X,M_m)$.

\begin{prop} \label{injtp is contractive}
Suppose that $X, X_1, Y$, and $Y_1$ are $p$-operator spaces. Given $p$-complete contractions $\varphi:X\to X_1$ and $\psi:Y\to Y_1$, the mapping
$$\varphi\otimes \psi:X\otimes Y \to X_1\otimes Y_1$$
extends to a $p$-complete contraction
$$\varphi\otimes \psi:X\injtp Y \to X_1\injtp Y_1.$$
\end{prop}

\begin{proof}
Since $\varphi\otimes \psi=(id_{X_1}\otimes \psi)\circ(\varphi\otimes id_Y)$, it suffices to show that $\varphi\otimes id_Y$ and $id_{X_1}\otimes \psi$ extend to  $p$-complete contractions. Let $u=[u_{ij}]\in M_n(X\otimes Y)$. Let us write $u_{ij}=\sum_k^{N_{ij}}x_k^{ij}\otimes y_k^{ij}$ for each $u_{ij}$. Since
$$(\varphi\otimes id_Y)_n(u)=\left[ \sum_k^{N_{ij}} \varphi(x_k^{ij})\otimes y_k^{ij} \right]\in M_n(X_1\otimes Y),$$ from (\ref{defi of injtp}) it follows that
$$\|(\varphi\otimes id_Y)_n(u)\|_{\vee_p}=\sup\left\{\left\|\left[ \sum_{k=1}^{N_{ij}} g_{st}(\varphi(x_k^{ij}))y_k^{ij}\right]\right\|_{M_{mn}(Y)}:m\in \mathbb{N}, g=[g_{st}]\in M_m(X_1')_1\right\}.$$
Define $h_{st}=g_{st}\circ \varphi$ for $1\leq s, t\leq m$, then $h=[h_{st}]=g\circ \varphi \in M_m(X')_1$ and we have
$$\|(\varphi\otimes id_Y)_n(u)\|_{\vee_p}\leq \|u\|_{\vee_p}.$$
To show that $id_{X_1}\otimes \psi$ is also $p$-completely contractive, let $v=[v_{ij}]\in M_n(X_1\otimes Y)$. Writing $v_{ij}=\sum_k^{N_{ij}}w_k^{ij}\otimes y_k^{ij}$, we have
$$    \|v\|_{\vee_p}=\sup\left\{\left\|\left[ \sum_{k=1}^{N_{ij}} f_{st}(w_k^{ij})y_k^{ij}\right]\right\|_{M_{mn}(Y)}:m\in \mathbb{N}, f=[f_{st}]\in M_m(X_1')_1\right\}.
$$
On the other hand,
\begin{align}
\begin{split}
\|(id_{X_1}\otimes \psi)_n(v)\|_{\vee_p}=&\sup\left\{\left\|\left[ \sum_{k=1}^{N_{ij}} f_{st}(w_k^{ij})\psi(y_k^{ij})\right]\right\|_{M_{mn}(Y_1)}:m\in \mathbb{N}, f=[f_{st}]\in M_m(X_1')_1\right\}\\
=&\sup\left\{\left\|\psi_{mn}\left(\left[ \sum_{k=1}^{N_{ij}} f_{st}(w_k^{ij})y_k^{ij}\right]\right)\right\|_{M_{mn}(Y_1)}:m\in \mathbb{N}, f=[f_{st}]\in M_m(X_1')_1\right\}\\
\leq&\|\psi\|_{pcb} \|v\|_{\vee_p}.
\end{split}
\end{align}
\end{proof}

\begin{remark}\label{injective tensor norm description}\quad
\begin{enumerate}
\item By definition of the Banach space injective tensor product $\overset{\epsilon}{\otimes}$, we have
    $$\|u\|_{\epsilon} = \|u\|_{\mathcal{B}(X',Y)} \leq \|u\|_{\cbp(X',Y)} = \|u\|_{\vee_p}$$ for every $u\in X\otimes Y$.
\item \label{part b} Let $u\in \mathbb{M}_n(X\otimes Y)$. If $Y \subseteq \mathcal{B}(L_p(\nu))$ for some measure $\nu$, then by Definition \ref{definition of itp} and \cite[Theorem 4.3, Proposition 4.4]{Daws}
\begin{eqnarray*}
\|u\|_{\vee_p}&=&\sup\{\|\psi(\varphi_{st}(u_{ij}))\|_{M_{rmn}}:m,k\in \mathbb{N}, \varphi=[\varphi_{st}]\in M_m(X')_1,\psi\in M_k(Y')_1\}\\
&=&\sup\{\|(\varphi\otimes \psi)_n(u)\|:m,k\in\mathbb{N}, \varphi\in M_m(X')_1, \psi\in M_k(Y')_1 \}.
\end{eqnarray*}
\item \label{injtp is comm} Let $F:X\otimes Y\to Y\otimes X$ denote the ``flip", that is, $F(\sum x_i\otimes y_i)=\sum y_i\otimes x_i$. If $Y \subseteq \mathcal{B}(L_p(\nu))$ for some measure $\nu$, then by (\ref{part b}) above, for every $u\in \mathbb{M}_n(X\otimes Y)$, we get
$$\|u\|_{\vee_p}=\sup\{\|(\varphi\otimes \psi)_n(u)\|:m,k\in\mathbb{N}, \varphi\in M_m(X')_1, \psi\in M_k(Y')_1 \}.
$$
On the other hand, if $X\subseteq \mathcal{B}(L_p(\mu))$ for some measure $\mu$ as well, then
$$\|F_n(u)\|_{\vee_p}=\sup\{\|(\psi\otimes\varphi)_n(F_n(u))\|:m,k\in\mathbb{N}, \varphi\in M_m(X')_1, \psi\in M_k(Y')_1 \}$$
and it follows that $X\injtp Y= Y\injtp X$ $p$-completely isometrically.
\item \label{MrMs equals Mrs} $M_r\injtp M_s$ is $p$-completely isometrically isomorphic to $M_{rs}$. This follows immediately from \cite[Theorem 3.2]{ALR}.
\end{enumerate}
\end{remark}

At this moment, we do not know whether the $p$-operator space injective tensor product is injective, that is, if $u:X\to \tilde{X}$ and $v:Y\to \tilde{Y}$ are $p$-completely isometric injections, then we do not know whether $u\otimes v$ always extends to a $p$-completely isometric injection $u\otimes v:X\injtp Y\to \tilde{X}\injtp \tilde{Y}$. But if we assume that all the $p$-operator spaces under consideration are on $L_p$ space, then we can show that $u\otimes v:X\injtp Y\to \tilde{X}\injtp \tilde{Y}$ is a $p$-complete isometry as in the following proposition. This fact supports that the terminology $p$-injective tensor product is still reasonable.

\begin{prop} \label{inj is inj provided Lp} Let $\mu_1,\mu_2$ be measures. For $i=1,2$, suppose $X_i\subseteq Y_i \subseteq \mathcal{B}(L_p(\mu_i))$. Then
$$X_1\injtp X_2 \subseteq Y_1\injtp Y_2$$ $p$-completely isometrically.
\end{prop}

\begin{proof}
For $i=1,2$, let $\varphi_i:X_i\hookrightarrow Y_i$ denote the ($p$-completely isometric) inclusion. Since $\varphi_1\otimes \varphi_2=(\varphi_1\otimes id_{Y_2})\circ(id_{X_1}\otimes \varphi_2)$, by Remark \ref{injective tensor norm description} (\ref{injtp is comm}) above, it suffices to show that
$$id_{X_1}\otimes \varphi_2:X_1\injtp X_2 \to X_1\injtp Y_2$$
is $p$-completely isometric. Note that the following diagram commutes:
$$\xymatrix{X_1\injtp X_2 \ar@{^(->}[d] \ar[r]^{id_{X_1}\otimes \varphi_2} & X_1\injtp Y_2 \ar@{^(->}[d] \\
\cbp(X_1',X_2)  \ar@{^(->}[r] & \cbp(X_1',Y_2)
}$$
Since $X_1\injtp X_2\subseteq \cbp(X_1',X_2)$, $X_1\injtp Y_2 \subseteq \cbp(X_1',Y_2)$, and
$\cbp(X_1',X_2)\subseteq \cbp(X_1',Y_2)$ $p$-completely isometrically, we conclude that $id_{X_1}\otimes \varphi_2$ is $p$-completely isometric.
\end{proof}

\section{Conditions $C'_p$, $C''_p$, and $C_p$ for $p$-Operator Spaces}
In this section, we define conditions  $C'_p$, $C''_p$, and $C_p$ for $p$-operator spaces and prove the main result. Throughout the section, $\mu$ and $\nu$ will denote measures.

\begin{lemma}\label{canonical map} Let $V$ and $W$ be $p$-operator spaces. Then the bilinear mapping
$$\tilde{\Psi}:V'\times W'\to (V\itp W)',\qquad (f,g)\mapsto f\otimes g$$ is jointly $p$-completely contractive and hence the canonical mapping $\Psi:V'\ptp W'\to (V\itp W)'$ is $p$-completely contractive.
\end{lemma}

\begin{proof} We identify $[f_{ij}]\in M_r(V')$ with an operator $F\in
\cbp(V,M_r)$, and likewise $[g_{kl}]\in M_s(W')$ with $G\in \cbp(W,M_s)$. We
have the identification $M_{rs}((V\injtp W)')=\cbp(V\injtp W,
M_{rs})$. Let $H$ be the map $[f_{ij}\otimes g_{kl}]:V\injtp W\to
M_{rs}$. Then by Proposition \ref{injtp is contractive} and Remark \ref{injective tensor norm description} (\ref{MrMs equals Mrs}) we have the commutative diagram
$$\xymatrix{V\injtp W \ar[d]_{F\otimes G} \ar[r]^{H} & M_{rs} \\
M_r\injtp M_s  \ar[ur]_\cong }$$
with $\|F\otimes G\|_{pcb}\leq\|F\|_{pcb}\|G\|_{pcb}$, and it follows that $\|[f_{ij}\otimes g_{kl}]\|=\|H\|_{pcb}\leq
\|F\otimes G\|_{pcb}\leq\|F\|_{pcb}\|G\|_{pcb}$ as required.
\end{proof}


\begin{lemma} Let $V$ and $W$ be $p$-operator spaces. Then $\|\cdot\|_{\vee_p}$ is a subcross
matrix norm. In particular, for every $u\in M_n(V\otimes W)$, we
have $\|u\|_{\vee_p}\leq \|u\|_{\wedge_p}$.
\end{lemma}

\begin{proof}
Just to fix notation, we identify $M_r(V)\otimes M_q(W)$ with $M_{rq}(V\otimes W)$ by $(v_{ij})\otimes (w_{kl})\mapsto (v_{ij}\otimes w_{kl})_{(i,k),(j,l)}$ where we have the ordering $(1,1)\leq (1,2)\leq \cdots \leq (1,q)\leq (2,1)\leq \cdots \leq (r,q)$. Hence $I_r\otimes w \in M_r\otimes M_q(W)=M_{rq}(W)$ is identified with a  block matrix in $M_r(M_q(W))$ which has $r$ copies of $w$ down the diagonal and $0$ elsewhere. Applying axiom $\mathcal{D}_\infty$ repeatedly hence shows that $\|I_r\otimes w\|_{rq}=\|w\|_q$. Then, for $\alpha\in M_r$, the matrix $\alpha\otimes w \in M_r\otimes M_q(W)=M_{rq}(W)$ is the product $(\alpha\otimes I_q)(I_r\otimes w)$ which has norm at most $\|\alpha\|_r\|w\|_q$ by axiom $\mathcal{M}_p$. Now let $v\in M_r(V)$ and $w\in M_q(W)$, and consider $v\otimes w\in M_{rq}(V\injtp W)$. This tensor induces the operator $T\in \cbp(V', M_{rq}(W))$ given by $T(f)=(f(v_{ij})w_{kl})_{(i,k),(j,l)}=(f(v_{ij}))\otimes w$. For $f=(f_{ab})\in M_n(V')$, we see that $T_n(f)=\llangle f, v \rrangle \otimes w\in M_{nrq}(W)$, which by the previous paragraph has norm at most $\|\llangle f, v \rrangle\|_{nr}\|w\|_q\leq \|f\|_n\|v\|_r\|w\|_q$. Hence $\|T\|_{pcb}\leq \|v\|_r\|w\|_q$ as required.
\end{proof}


Let $V$ and $W$ be $p$-operator spaces and fix $\varphi\in(V\itp W)'$. For $v_0\in V$,
we define a bounded linear functional $_{v_0}\varphi$ on $W$ by
$$_{v_0}\varphi(w)=\varphi(v_0\otimes w), \qquad w\in W.$$

In general, when $v_0=[v_{ij}]\in M_r(V)$ and
$\varphi=[\varphi_{kl}]\in M_n((V\itp W)')$, we define
$_{v_{0}}\varphi = [_{v_{ij}}\varphi_{kl}] \in M_{rn}(W')$. Similarly, for $w_0\in W$, we define $\varphi_{w_0}\in V'$ by
$$\varphi_{w_0}(v)=\varphi(v\otimes w_0), \qquad v\in V.$$

As in $_{v_{0}}\varphi$ above, we can extend the definition of
$\varphi_{w_0}$ for $w_0\in M_r(W)$ and $\varphi\in M_n((V\itp W)')$.
Define a linear map $\Phi_{V,W}^R:V\otimes W{''}\to (V\itp W){''}$
by
$$\Phi_{V,W}^R(v\otimes w{''})(\varphi)=\langle _v\varphi, w{''} \rangle_{W',W{''}}, \qquad v\in V, \quad w{''}\in W{''}, \quad \varphi \in (V\itp W)'.$$

Similarly, define a linear map $\Phi_{V,W}^L:V{''}\otimes W\to
(V\itp W){''}$ by
$$\Phi_{V,W}^L(v{''}\otimes w)(\varphi)=\langle \varphi_w, v{''} \rangle_{V',V{''}}, \qquad v{''}\in V{''}, \quad w\in W, \quad \varphi \in (V\itp W)'.$$

\begin{lemma}\label{universal property}
The map $\Phi_{V,W}^R$ (respectively, $\Phi_{V,W}^L$) defined above extends
to a $p$-completely contractive map $\Phi_{V,W}^R:V\ptp W{''}\to
(V\itp W){''}$ (respectively, $\Phi_{V,W}^L:V{''}\ptp W\to (V\itp
W){''}$).
\end{lemma}

\begin{proof}
Consider the bilinear map $\Phi:V\times W{''}\to (V\itp W){''}$
given by $$(v,w{''})\mapsto (\varphi\mapsto\langle
_v\varphi,w{''}\rangle_{W',W{''}}),$$ 
then we get
$$\Phi_{r;s}:M_r(V)\times M_s(W{''})\to M_{rs}((V\itp W){''}), \qquad ([v_{ij}],[w_{kl}{''}])\mapsto [\Phi(v_{ij},w_{kl}{''})]$$
and
$$\|[\Phi(v_{ij},w_{kl}{''})]|=\sup_n\left\{\|\llangle
\Phi_{r;s}(v,w{''}), \varphi \rrangle\|:\varphi\in
M_n((V\itp W)'), \|\varphi\|\leq 1\right\}.$$ Since $\llangle
\Phi_{r;s}(v,w{''}), \varphi \rrangle =\llangle
_v\varphi,w{''} \rrangle$, we have
$$\|\llangle
\Phi_{r;s}(v,w{''}), \varphi \rrangle \|=\|\llangle
_v\varphi,w{''} \rrangle \|\leq
\|_v\varphi\|_{M_{rn}(W')}\cdot\|w{''}\|_{M_s(W{''})}$$ and the
result follows because $\injtp$ is a subcross matrix norm and hence
$$\begin{array}{rcl}\|_v\varphi\|_{M_{rn}(W')}&=&\sup_m\left\{\|\llangle  _v\varphi, w\rrangle \|_{M_{rnm}}:w\in M_m(W),\|w\|\leq
1\right\} \\
&=& \sup_m\left\{\|\llangle  \varphi, v\otimes
w\rrangle \|_{M_{rnm}}:w\in M_m(W),\|w\|\leq 1\right\}\\
&\leq & \|\varphi\|\cdot\|v\|\\
&\leq & \|v\|.
\end{array}$$
\end{proof}

\begin{remark}\label{alpha in place of inj} Let $\alpha$ be a general subcross matrix norm.
\begin{enumerate}
\item We have a natural $p$-complete contraction $V\projtp W\to V\otimes_\alpha W$ and the adjoint gives a contraction $(V\otimes_\alpha W)'\to \cbp(V,W')\subseteq \mathcal{B}(V,W')$ given by
$$\varphi\mapsto L_\varphi,\quad \langle L_\varphi(v),w \rangle=\varphi(v\otimes w), \quad \varphi\in (V\otimes_\alpha W)'\quad v\in V,\quad w\in W.$$
\item \label{part b} Using the natural $p$-complete contraction $V\projtp W\to V\otimes_\alpha W$, each member in $(V\otimes_\alpha W)'$ can be regarded as a member in $(V\projtp W)'$.
\item \label{part c} We can define $\Phi_{V,W}^R:V\otimes W{''}\to (V\otimes_\alpha W){''}$ and $\Phi_{V,W}^L:V{''}\otimes W\to (V\otimes_\alpha W){''}$ for a general subcross norm $\alpha$ and Lemma \ref{universal property} remains valid if $\injtp$ is replaced by $\otimes_\alpha$.
\end{enumerate}
\end{remark}

Let $\Psi:V'\ptp W'\to (V\itp W)'$ denote the canonical map, and
consider the following commutative diagram
\begin{center}
\begin{minipage}{9cm}
$$\xymatrix{V\otimes W{''} \ar@{=}[d] \ar[r]^{\Phi_{V,W}^R} & (V\itp W){''} \ar[r]^{\Psi'} & (V'\ptp W')' \ar@{=}[d] \\
\mathcal{CB}_{p,F}^\sigma(V',W{''}) \ar@{^{(}->}[rr]^{\iota} &  &
\mathcal{CB}_p(V',W{''}) }$$
\end{minipage},
\end{center}
 where
$\mathcal{CB}_{p,F}^\sigma(V',W{''})$ denotes the space of all
weak$^*$-continuous $p$-completely bounded finite rank maps from
$V'$ to $W{''}$ and $\iota$ denotes the inclusion map. This commutative
diagram shows that $\Phi_{V,W}^R$ is one-to-one, so one can equip
$V\otimes W{''}$ with the $p$-operator space norm inherited from
$(V\itp W){''}$, which will be denoted by, following the notation
in \cite{EffrosRuan}, $V\itp\!:\!W{''}$. We say that $V$ satisfies
\textit{condition} $C'_p$ (or $V$ has \textit{property} $C'_p$) if this
induced norm coincides with the $p$-operator space injective tensor product norm for every $W\subseteq \mathcal{B}(L_p(\nu))$.

Similarly, the following diagram
$$\xymatrix{V{''}\otimes W \ar@{=}[d] \ar[r]^{\Phi_{V,W}^L} & (V\itp W){''} \ar[r]^{\Psi'} & (V'\ptp W')' \ar@{=}[d] \\
\mathcal{CB}_{p,F}^\sigma(W',V{''}) \ar@{^{(}->}[rr]^{\iota} &  &
\mathcal{CB}_p(W',V{''}) }$$ is also commutative, $\Phi_{V,W}^L$
is one-to-one, and one can hence equip $V{''}\otimes W$ with the
$p$-operator space norm inherited from $(V\itp W){''}$, which will
be denoted by $V{''}\!:\!\itp  W$. We say that $V$ satisfies
\textit{condition} $C''_p$ (or $V$ has \textit{property} $C''_p$) if
this induced norm coincides with the injective tensor
product norm for every $W\subseteq \mathcal{B}(L_p(\nu))$.

In order to define \textit{condition} $C_p$ for $p$-operator spaces, we need the natural map from $V{''}\otimes W{''}$ to $(V\injtp W)''$. To do this, let $\alpha$ be a general subcross matrix norm on $V\otimes W$ and consider the diagram
\begin{equation}\label{diagram} \xymatrix{ & (V\ptp W{''}){''} \ar[rd]^{(\Phi^R_{V,W}){''}} & & \\
V{''}\otimes W{''} \ar[ru]^{\Phi^L_{V,W{''}}}
\ar[rd]_{\Phi^R_{V{''},W}} & & (V\otimes_\alpha W)^{''''}
\ar[r]^{P} & (V\otimes_\alpha
W){''}\quad , \\
& (V{''}\ptp W){''} \ar[ru]_{(\Phi^L_{V,W}){''}} & &}\end{equation} where $P$
is the restriction mapping and $(\Phi^R_{V,W}){''}$ and $(\Phi^L_{V,W}){''}$ are from Remark \ref{alpha in place of inj} (\ref{part c}). 

Consider the following $p$-complete
contraction:
$$\begin{CD}(V\ptp W)' \cong \cbp(V,W') @>\text{adj}>> \cbp(W{''},V')\cong (V\ptp W{''})' \end{CD}.$$
For $\varphi \in (V\ptp W)'$, let
$\varphi^{\wedge}\in(V\ptp W{''})'$ denote the image of $\varphi$
under this map. Then we have
$$\varphi^{\wedge}(v\otimes w{''})=\langle _v\varphi, w{''}\rangle_{W',W{''}}=\Phi_{V,W}^R(v\otimes w{''})(\varphi), \qquad v\in V, \quad w{''}\in W{''}.$$
Moreover, $\varphi^{\wedge}$ is weak*-continuous in the second
variable. Similarly, we also consider the $p$-complete contraction
$$\begin{CD}(V\ptp W)' \cong \cbp(W,V') @>\text{adj}>> \cbp(V{''},W') \cong (V{''}\ptp W)' \end{CD}$$
and define ${^\wedge\varphi}$, and then we get that
$${^\wedge\varphi}(v{''}\otimes w)=\langle \varphi_w,v{''}\rangle_{V',V{''}}=\Phi^L_{V,W}(v{''}\otimes w)(\varphi), \qquad v{''}\in V{''}, \quad w\in W, $$ and that ${^\wedge\varphi}$ is weak*-continuous in the first
variable.

\begin{remark} \label{hat of varphi} Let $\alpha$ be a general subcross matrix norm. By Remark \ref{alpha in place of inj} (\ref{part b}), we can still define $\varphi^{\wedge}\in(V\ptp W{''})'$ for any $\varphi \in (V\otimes_\alpha W)'$. Similarly, we can define  ${^\wedge\varphi}\in (V{''}\ptp W)' $ for any $\varphi \in (V\otimes_\alpha W)'$. 
\end{remark}

The next result follows by Remarks \ref{alpha in place of inj} and \ref{hat of varphi}, and the same argument as in the proof of \cite[Theorem 1]{Han}.
\begin{thm} \label{equivalences}
Let $V$ and $W$ be $p$-operator spaces. Let $\alpha$ be a subcross matrix norm on
$V\otimes W$ and denote by $V\otimes_\alpha W$ the resulting normed
space. Then the following are equivalent.
\begin{enumerate}
\item \label{equiv 1} There exists a separately weak*-continuous extension
$$\Phi:V{''}\otimes W{''}\to (V\otimes_\alpha W){''}$$
of the natural inclusion $\iota:V\otimes W\to (V\otimes_\alpha
W){''}$.
\item \label{equiv 2} The following diagram commutes
$$\xymatrix{ & (V\ptp W{''}){''} \ar[rd]^{(\Phi^R_{V,W}){''}} & & \\
V{''}\otimes W{''} \ar[ru]^{\Phi^L_{V,W{''}}}
\ar[rd]_{\Phi^R_{V{''},W}} & & (V\otimes_\alpha W)^{''''}
\ar[r]^{P} & (V\otimes_\alpha
W){''}\quad . \\
& (V{''}\ptp W){''} \ar[ru]_{(\Phi^L_{V,W}){''}} & &}$$
\item \label{equiv 3} For every $\varphi\in (V\otimes_\alpha W)'$, two functionals
$({^\wedge\varphi})^\wedge$ and $^\wedge(\varphi^\wedge)$ coincide
on $V{''}\otimes W{''}$.
\item \label{equiv 4} For every $\varphi\in (V\otimes_\alpha W)'$, $L_\varphi:V\to
W'$ is weakly compact, where $\langle L_\varphi(v),
w\rangle=\varphi(v\otimes w)$, $v\in V$, $w\in W$.
\end{enumerate}
\end{thm}

\begin{thm} \label{itp works}Let $V\subseteq \mathcal{B}(L_p(\mu))$ and $W\subseteq \mathcal{B}(L_p(\nu))$. For every $\varphi\in(V\itp W)'$,
$L_\varphi$ is weakly compact, where $L_\varphi$ is as in Theorem \ref{equivalences} (\ref{equiv 4}).
\end{thm}

\begin{proof}
Without loss of generality, we may assume
$\|\varphi\|(=\|\varphi\|_{pcb})\leq 1$. Let $\Phi^V$ (respectively $\Phi^W$) denote the embedding $\Phi^V:V\hookrightarrow \mathcal{B}(L_p(\mu))$ (respectively, $\Phi^W:W\hookrightarrow \mathcal{B}(L_p(\nu))$). By Proposition \ref{inj is inj provided Lp} and \cite[Theorem 3.2]{ALR}, we have $p$-completely isometric embeddings $$V\itp W\hookrightarrow \mathcal{B}(L_p(\mu))\itp \mathcal{B}(L_p(\nu)) \hookrightarrow \mathcal{B}(L_p(\mu\times \nu)).$$
Consider the diagram below:
$$\xymatrix{V\itp W \ar[rrr]^\varphi \ar@{^{(}->}[d]_{\Phi^V\otimes \Phi^W} & & & \mathbb{C} \\ \mathcal{B}(L_p(\mu))\itp \mathcal{B}(L_p(\nu)) \ar@{^{(}->}[d] & & & \\ \mathcal{B}(L_p(\mu\times \nu)) \ar@{-->}[rrruu]_{\tilde{\varphi}} & & &
}$$ By Hahn-Banach Theorem, $\varphi$ extends to
$\tilde{\varphi}:\mathcal{B}(L_p(\mu\times \nu)) \to \mathbb{C}$. Applying
the same technique as in the proof of \cite[Theorem 3.6]{ALR}, we can find a measure
space $(\Omega,\Sigma,\theta)$ together with two vectors $\xi\in
L_p(\theta)$, $\eta\in L_{p'}(\theta)$, and a unital $p$-completely
contractive homomorphism $\pi:\mathcal{B}(L_p(\mu\times \nu))\to
\mathcal{B}(L_p(\theta))$ such that $\tilde{\varphi}(\cdot)=\langle
\pi(\cdot)\xi,\eta \rangle$.\\
Define $T:\mathcal{B}(L_p(\mu))\to \mathcal{B}(L_p(\nu))'$ by
$$\langle T(x),y\rangle = \tilde{\varphi}(x\otimes y),\qquad x\in \mathcal{B}(L_p(\mu)),\quad y\in \mathcal{B}(L_p(\nu)).$$
Then it is easy to check that the following diagram is commutative:
$$\xymatrix{V\ar[r]^{L_\varphi} \ar@{^{(}->}[d]_{\Phi^V} & W' \\
\mathcal{B}(L_p(\mu))\ar[r]^T & \mathcal{B}(L_p(\nu))'\ar[u]_{(\Phi^W)'} }$$ Define
$R:\mathcal{B}(L_p(\mu))\to L_p(\theta)$ and $S:\mathcal{B}(L_p(\nu))\to
L_{p'}(\theta)$ by
$$R(x)= \pi(x\otimes 1)\xi,\quad x\in \mathcal{B}(L_p(\mu)),\qquad \text{and} \qquad S(y)= (\pi(1\otimes y))'\eta,\quad y\in \mathcal{B}(L_p(\nu)),$$
then the diagram
$$\xymatrix{\mathcal{B}(L_p(\mu))\ar[rr]^T \ar[rd]_R & & \mathcal{B}(L_p(\nu))' \\
& L_p(\theta)\ar[ru]_{S'} & }$$ is commutative, because
$$\langle S'R(x),y\rangle=\langle R(x),S(y) \rangle = \langle \pi(x\otimes 1)\xi,(\pi(1\otimes y))'\eta \rangle=
\langle\pi(x\otimes y)\xi,\eta\rangle=\tilde{\varphi}(x\otimes
y)=\langle T(x),y\rangle.$$ Combining these two commutative
diagrams, we finally have $L_\varphi=(\Phi^W)'S'R\Phi^V$, that is,
$L_\varphi$ is factorized through a reflexive Banach space
$L_p(\theta)$, so $L_\varphi$ is a weakly compact operator
\cite[Propositions 3.5.4 and 3.5.11]{Megginson}.
\end{proof}

\begin{coro} \label{corollary} Let $V,W$ be $p$-operator spaces on $L_p$ space. Then there exists a (necessarily unique) separately weak*-continuous
extension
$$\Phi:V{''}\otimes W{''}\to (V\itp W){''}$$
of the natural inclusion $\iota:V\otimes W\to (V\itp W){''}$.
\end{coro}

\begin{proof}
Combine Theorem \ref{equivalences} and Theorem \ref{itp works}.
Uniqueness follows from separate weak*-continuity.
\end{proof}

Now we are ready to define condition $C_p$ for $p$-operator spaces.
Let $\Phi$ be as in Corollary \ref{corollary}. The following
commutative diagram
$$\xymatrix{V{''}\otimes W{''} \ar@{=}[d] \ar[r]^{\Phi} & (V\itp W){''} \ar[r]^{\Psi'} & (V'\ptp W')' \ar@{=}[d] \\
\mathcal{CB}_{p,F}^\sigma(V',W{''}) \ar@{^{(}->}[rr]^{\iota} &  &
\mathcal{CB}_p(V',W{''}) }$$ shows that $\Phi$ is injective. Thus
we can equip $V{''}\otimes W{''}$ with the $p$-operator space
structure induced by $\Phi$, which will be denoted by
$V{''}\!:\!\itp\!:\!W{''}$. We say that $V\subseteq \mathcal{B}(L_p(\mu))$ satisfies \textit{condition} $C_p$ (or has
\textit{property} $C_p$) if the map $\Phi$ is isometric with respect to
the injective tensor product norm for every $W\subseteq \mathcal{B}(L_p(\nu))$.

\begin{prop}
Suppose that $V\subseteq \mathcal{B}(L_p(\mu))$. Then $V$ satisfies
condition $C_p$ if and only if $V$ satisfies both condition $C'_p$ and
$C''_p$.
\end{prop}

\begin{proof}
Suppose that $V$ satisfies condition $C_p$ and $W\subseteq \mathcal{B}(L_p(\nu))$. By Proposition \ref{inj is inj provided Lp} and \cite[Theorem 4.3]{Daws}, we have a $p$-completely isometric
embedding $V\itp  W{''} \subseteq V{''}\itp
W{''}$ and the bottom row in the following commutative diagram
$$\xymatrix{V\itp\!:\!W{''}\ar[r]\ar@{^{(}->}[d] &
V\itp W{''}\ar@{^{(}->}[d]\\
V{''}\!:\!\itp\!:\!W{''}\ar[r] & V{''}\itp W{''}
}$$ is isometric. Therefore the top row is also isometric and hence
$V$ satisfies condition $C'_p$. That $V$ satisfies condition $C''_p$ can
be proved using a similar argument.

On the other hand, if $V$ satisfies condition $C''_p$, we get
$$V{''}\itp W{''}=V{''}:\itp \!:\!W{''}\hookrightarrow (V\itp W{''}){''}.$$
If $V$ also satisfies condition $C'_p$, then
$$V\itp W{''}=V\itp \!:\!W{''}\hookrightarrow (V\itp W){''},$$
and hence we have isometric inclusion
$$V{''}\itp W{''}\hookrightarrow (V\itp W){''''}.$$
Since $V{''}\itp W{''}\subset(V\itp W){''}$
and $(V\itp W){''}\hookrightarrow
(V\itp W){''''}$ isometrically, the inclusion
$V{''}\itp W{''}\subseteq (V\itp W){''}$ must
be isometric.
\end{proof}

\section{Acknowledgement}
The author would like to thank the reviewer for his/her valuable comments, especially the ones that led to Remarks \ref{alpha in place of inj} and \ref{hat of varphi}, to improve the quality of the paper.

\bibliography{jungjinthesisbib}
\bibliographystyle{alpha}

\end{document}